\documentclass{article}

\usepackage[french]{babel}
\usepackage[utf8x]{inputenc}
\usepackage{amssymb, mathrsfs, amsthm, amsmath, mdwlist, enumerate, graphicx}
\usepackage[all]{xy}
\usepackage{url}

\newcommand*{\TakeFourierOrnament}[1]{{%
\fontencoding{U}\fontfamily{futs}\selectfont\char#1}}
\newcommand*{\danger}{\TakeFourierOrnament{66}}

\renewcommand{\AA}{\mathbb{A}}
\newcommand{\cI}{\mathcal{I}}

\newcommand{\PP}{\mathbb{P}}
\newcommand{\ZZ}{\mathbb{Z}}
\newcommand{\QQ}{\mathbb{Q}}
\newcommand{\OO}{\mathcal{O}}
\newcommand{\zom}{\mathbb{Z} / m}

\newcommand{\op}{\ensuremath{^\mathit{op}}}

\newcommand{\eff}{\ensuremath{^\mathit{eff}}}
\newcommand{\gm}{\ensuremath{_\mathit{gm}}}
\newcommand{\equi}{\ensuremath{_\mathit{equi}}}

\newcommand{\singh}{\ensuremath{^\mathit{sing}}}
\newcommand{\singc}{\ensuremath{_\mathit{sing}}}

\newcommand{\et}{\ensuremath{_\mathit{et}}}
\newcommand{\tqfh}{\ensuremath{\mathit{qfh}}}

\newcommand{\PreShv}{\textsf{PreShv}}
\newcommand{\Ab}{\textsf{Ab}}
\newcommand{\Shv}{\textsf{Shv}}

\newcommand{\Sm}{\textsf{Sm}}
\newcommand{\Sch}{\textsf{Sch}}
\newcommand{\Cor}{\textsf{Cor}}

\newcommand{\Comp}{\textsf{Comp}}

\newcommand{\zpi}{\ZZ[\tfrac{1}{p}]}
\newcommand{\opi}{[\tfrac{1}{p}]}

\newtheorem{theo}{Théorème}
\newtheorem{prop}[theo]{Proposition}
\newtheorem*{prop*}{Proposition}
\newtheorem*{coro*}{Corollaire}
\newtheorem{lemm}[theo]{Lemme}
\newtheorem{coro}[theo]{Corollaire}
\theoremstyle{theorem}

\theoremstyle{definition}
\newtheorem{defi}[theo]{Définition}
\newtheorem{rema}[theo]{Remarque}

\DeclareMathOperator{\id}{id} 
\DeclareMathOperator{\Tot}{Tot}
\DeclareMathOperator{\uhom}{\underline{hom}}

\DeclareMathOperator{\length}{length}
\DeclareMathOperator{\Tr}{Tr}
\DeclareMathOperator{\RHom}{RHom}

\newcommand{\viso}{\rotatebox{270}{$\cong$}}

\newcommand{\isoCHsing}{\Phi_\mathit{s}^\mathit{CH}}

\newcommand{\isoCHsingcoh}{\Phi_\mathit{s, et}}
\newcommand{\isosingmot}{\Phi^\mathit{s}_\mathit{M}}

\newcommand{\RIvo}{\Phi^\mathit{M}_\mathit{et}}
\newcommand{\isoCHmot}{\Phi^\mathit{CH}_\mathit{M}}
\newcommand{\duas}{\mathcal{D}_s}
\newcommand{\duaet}{{\mathcal{D}\et}}
\newcommand{\trCH}{{\mathit{Tr}^{\mathit{CH}}}}
\newcommand{\trEt}{{\mathit{Tr}^{\mathit{et}}}}
\newcommand{\trM}{{\mathit{Tr}^{\mathit{M}}}}
\newcommand{\PhiR}{{\Phi_R}}
\newcommand{\PhiS}{{\Phi_{\mathit{Sus}}}}
\newcommand{\PhiBL}{{\Phi_{\mathit{BL}}}}

\title{Un isomorphisme de Suslin}
\author{Shane Kelly}

\begin{document}

\maketitle


\begin{abstract}
Dans cette note on constate qu'on peut enlever l'hypothèse de la résolution des singularités de l'isomorphisme construit par Suslin entre la cohomologie étale à supports compacts et les groupes de Chow supérieurs de Bloch. On démontre de plus que l'on peut obtenir cet isomorphisme à partir de la réalisation étale d'Ivorra. 
\end{abstract}





\section{Introduction}

Un premier but de ce note est d'observer qu'en remplaçent le théorème \cite[Theorem 4.1.2]{Voev00} de Voevodsky avec le résultat principal (\cite[Theorem 5.3.1]{Kel12}) de la thèse de l'auteur, on peut enlever l'hypothèse de la résolution des singularités dans le résultat principal de \cite{Sus00}.

\begin{theo}[{\cite[Theorem 5.6.1]{Kel12}, cf. \cite[Introduction]{Sus00}}] \label{theo:susMain}
Soit $k$ un corps algébriquement clos, $m$ un entier inversible dans $k$, $\pi_X: X \to k$ un morphisme quasi-projectif équidimensionnel de dimension $d$ et soit $j \geq 0$. Alors il existe un isomorphisme naturel
\[ \PhiS: H^{n - 2j}_c(X, \ZZ / m(-j)) \cong CH^{j + d}(X, n; \ZZ / m)^\# \]
où $H_c$ est la cohomologie étale à support compact, $CH$ sont les groupes supérieurs de Chow, et on utilise $(-)^\#$ pour indiquer $\hom_{Ab}(-, \QQ / \ZZ)$.
\end{theo}

Suslin définit cet isomorphisme à partir d'un isomorphisme $CH^{d}(X, n; \ZZ / m) \cong H_n^{sing}(X, \ZZ / m)$ (démontré dans \cite{Sus00}), un isomorphisme $H^{n}_c(X, \ZZ / m) \cong H^n\singc(X, \ZZ / m)$ (démontré dans \cite{SV96}), et l'isomorphisme canonique
\begin{equation} \label{equa:singPair}
H^n_{sing}(X, \ZZ / m) \cong H_n^{sing}(X, \ZZ / m)^\#.
\end{equation}

D'un autre côté, on peut définir un deuxième morphisme 
\[ \PhiR: H^{n - 2j}_c(X, \ZZ / m(-j)) \to CH^{j + d}(X, n; \ZZ / m)^\# \]
à partir du morphisme canonique 
\begin{equation} \label{equa:etalPair}
\hom(\ZZ / m, R\pi_{X!}(\ZZ / m)_X(p)[q]) \to \hom(R\pi_{X!}(\ZZ / m)_X(p)[q], \ZZ / m)^\#
\end{equation}
où les morphismes sont dans la catégorie dérivée $D_{et}(k, \ZZ / m)$ des faisceaux de $\ZZ / m$-modules sur le petit site étale de $k$ (cf. Définition~\ref{defi:phiIvo}).

La deuxième but de cette note est de démonter que $\PhiS$ est égal à $\PhiR$.

\begin{prop*}[{Proposition~\ref{prop:pairing}, Proposition~\ref{prop:susEqualIvo}}]
Sous les hypothèses du Théorème~\ref{theo:susMain}, l'isomorphisme
\[ H^n\singc(X, \ZZ / m) \cong \hom(\ZZ / m, R\pi_{X!}(\ZZ / m)_X[n]) \]
défini dans \cite{SV96}, et le morphisme
\[ H_n^{sing}(X, \ZZ / m) \to \hom(R\pi_{X!}(\ZZ / m)_X[n], \ZZ / m) \]
de la Définition~\ref{defi:phiIvo} (défini à partir du foncteur de réalisation), sont compatibles avec les accouplements canoniques (\ref{equa:singPair}) et (\ref{equa:etalPair}). Par conséquence, \mbox{$\PhiS$ est égal à $\PhiR$}.
\end{prop*}



\begin{paragraph}{Conventions.}
Pour un groupe abélien $A$, les éléments de $A^\#$ seront induits en général par un morphisme $A \to \zom$ et un choix d'injection $\zom \subset \QQ / \ZZ$. Donc, on choisit une fois pour tout une telle injection.
\end{paragraph}

\section{L'isomorphisme de Suslin}


\subsection{Les groupes de Chow supérieurs et l'homologie singulière}

Dans cette sous-section on rappelle la notation, le résultat principal de la Section 2 de \cite{Sus00}, et son extension au cas non-affine pour que le lecteur constate que l'hypothèse de le résolution des singularités n'est pas nécessaire si l'on remplace \cite[Theorem 4.1.2]{Voev00} avec \cite[Theorem 5.1.3]{Kel12} (reproduit en dessous).

%
Soit $k$ un corps et soit $X$ un $k$-schéma séparé équidimensionnel de type fini. Soit $z^i(X, n)$ le groupe libre abélien engendré par les sous-variétés de $\Delta_X^n$ de codimension $i$ tels que l'intersection avec chaque coface $\Delta^m_X \subset \Delta^n_X$ est de codimension $\geq i$ dans $\Delta^m_X$ (où $\Delta^n_S$ est le sous-schéma linéaire de $\AA^{n + 1}_S$ défini par $t_0 + \dots + t_n = 1$ et les cofaces sont les sous-variéties définies par $t_{j_1} = 0, \dots, t_{j_{n - m}} = 0$). 
Par définition \cite{Blo86}, les groupes $CH^i(X, n)$ sont les groupes d'homologie du complexe
$ \dots \stackrel{}{\to} z^i(X, 2) \stackrel{}{\to} z^i(X, 1) \stackrel{}{\to} z^i(X, 0) \to 0 \to \dots  $
où les différentielles sont $\sum_{j = 0}^n (-1)^j \partial^*_j$ 
 et le morphisme $\partial_j^*$ est l'intersection avec la coface $t_j = 0$.
Le complexe $z^i(X, \ast)$ possède un sous-complexe $z^i_{equi}(X, \ast)$ où $z^i_{equi}(X, n)$ est le sous-groupe libre %
engendré par les sous-variétés $V \subset \Delta^n_X$ tels que la projection vers $\Delta_k^n$ est un morphisme équidimensionel%
\footnote{Quand on écrit équidimensionnel de dimension relative $t$ on veut dire un morphisme de schémas $f: W \to V$ qui est de type fini, dominant, est pour lequel le fonction $\dim p^{-1} p(-)$ sur $W$ est constant et égal à $t$. } %
de dimension relative $d - i$. 
Le théorème technique principal de \cite{Sus00} est le suivant. Il est vrai sans aucune hypothèse sur le corps $k$!

\begin{theo}[{\cite[Theorem 2.1]{Sus00}}] \label{theo:affineMoving}
Soit $X$ un $k$-schéma affine équidimensionnel et $i \leq d = \dim X$. Alors, l'inclusion de complexes $z\equi^i(X, \ast) \hookrightarrow z^i(X, \ast)$ est un quasi-isomorphisme.
\end{theo}


Par $\Sch / k$ on veut dire la catégorie des $k$-schémas séparés de type fini et par $\Sm / k$ la sous-catégorie pleine des $k$-schémas lisses. Soit $X \in \Sch / k$ et soit $t \geq 0$. Suslin utilise la notation $z_t(X)$ pour le préfaiscau%
\footnote{Pour quelque détails sur ce préfaisceau, la catégorie $\Cor / k$ des correspondances (non-lisse), et préfaisceaux avec transferts voit la Section~\ref{sect:cyclesRelatifs}.} %
$z\equi(X, t)$ de \cite{SV}. 

Les groupes de (co)homologie singulière d'un préfaisceau $F$ sur $\Sch / k$ sont définis comme $H\singh_n(F) = H_n(C_\ast F)$ (où $C_\ast F$ est le complexe avec $C_nF = F(\Delta^n_k)$) et pour un groupe abélien $\Lambda$ comme
\[ H\singh_n(F, \Lambda) = H_n(C_\ast F \stackrel{L}{\otimes} \Lambda), \quad H\singc^n(F, \Lambda) = H^n(\RHom(C_\ast F, \Lambda)). \]

On remplace le théorème \cite[Theorem 4.1.2]{Voev00} de Voevodsky (qui apparait comme \cite[Theorem 3.1]{Sus00} dans l'article de Suslin) avec la version suivante qui ne suppose pas la résolution des singularités.


\begin{theo}[{\cite[Theorem 5.1.3]{Kel12}, cf. \cite[Theorem 4.1.2]{Voev00}}] \label{theo:kel12}
Soit $k$ un corps parfait de caractéristique exponentielle $p$. Soit $F$ un préfaisceau de $\zpi$-modules avec transferts tels que%
\footnote{On renvoie le lecteur qui \mbox{ne connait pas la cdh topologie au très accessible \cite[Section 5]{SV00}}.} %
$F_{cdh} = 0$. Alors, l'image de\footnote{Le complexe de préfaisceaux $\underline{C}_\ast F$ est la version préfaisceautique de $C_\ast F$ défini par \mbox{$\underline{C}_n F(Y) = F( \Delta^n_Y)$} (en fait, $C_\ast F$ n'est que le complexe des sections globales $(\underline{C}_\ast F)(k)$).} $\underline{C}_\ast F$ dans la catégorie des complexes de faisceaux de Nisnevich sur $\Sm / k$ est acyclique.
\end{theo}

Ce théorème implique tout de suite l'extension suivante.

\begin{theo}[{cf. \cite[Theorem 3.2]{Sus00}}]
Soit $k$ un corps parfait de caractéristique exponentielle $p$. Alors, pour tout $k$-schéma quasi-projectif équidimensionnel $X$ et tout $i \leq d = \dim X$, l'inclusion canonique de complexes 
\[ C_\ast(z_{d - i}(X))\opi = z\equi^i(X, \ast)\opi \to z^i(X, \ast)\opi \]
est un quasi-isomorphisme. Par conséquence, cette inclusion induit des isomorphismes
\begin{equation} \label{equa:CHsing}
\isoCHsing: H_n\singh(z_{d - i}(X), \ZZ / m) \stackrel{\sim}{\to} CH^{i}(X, n; \ZZ / m)
\end{equation}
pour tout $n \in \ZZ$ et tout $m$ premier à $p$.
\end{theo}

\begin{proof}
On rappelle la démonstration de \cite[Theorem 3.2]{Sus00} qui marche sans problème une fois que \cite[Theorem 3.1]{Sus00} (i.e., \cite[Theorem 4.1.2]{Voev00}) est remplacé par Théorème~\ref{theo:kel12}.

On travaille par récurrence sur $d = \dim X$. Si $d = 0$ les deux complexes sont égaux, donc considérons le cas $d > 0$ et assumons que le résultat est connu pour les variétés de dimension inférieure à $d$. Soit $Y \subset X$ un diviseur de Cartier effectif tel que $U = X - Y$ est affine. La suite de préfaisceaux
$ 0 \to z_{d - i}(Y) \to z_{d - i}(X) \to z_{d - i}(U) $
est exacte, et la cdh faisceau associé à $z_{d - i}(U) / z_{d - i}(X)$ est zéro (\cite[Theorem 4.2.9, Theorem 4.3.1]{SV00}). D'où, d'après le Théorème~\ref{theo:kel12}, l'application de $(\underline{C}_\ast(-))_{\textit{Nis}}(k) = C_\ast(-)$ résulte en un triangle distingué dans la catégorie dérivée des groupes abéliens. De l'autre côté, la suite $z^{i - 1}(Y, \ast) \to z^i(X, \ast) \to z^i(U, \ast)$ est aussi un triangle distingué (\cite[Theorem 3.3]{Blo86}). L'inclusion à gauche $C_\ast  z_{d - i}(Y)\opi \subseteq z^{i - 1}(Y, \ast)\opi$ est un quasi-isomorphisme par l'hypothèse de récurrence et l'inclusion à droite $C_\ast z_{d - i}(U)\opi \subseteq z^i(U, \ast)\opi$ est un quasi-isomorphisme par le Théorème~\ref{theo:affineMoving}. Donc l'inclusion au milieu $C_\ast z_{d - i}(X)\opi \subseteq z^i(X, \ast)\opi$ est aussi un quasi-isomorphisme.
\end{proof}

\subsection{La cohomologie étale, la cohomologie singulière, et les groupes de Chow supérieurs}

Désormais on fixe un $m$ premier à $p$ et on écrit $\Lambda = \ZZ / m$.

Dans cette sous-section on rappelle les matières de \cite[Section 4]{Sus00} et \cite{SV96} pour arriver à l'isomorphisme $CH^{d + j}(X, n; \Lambda)^\# \cong H_c^{n - 2j}(X, \Lambda(-j))$. Il sera la composition d'une suite d'isomorphismes.

En regardant les définitions, on voit qu'il y a un accouplement canonique entre $H^n\singc(-, \Lambda)$ et $H_n\singh(-, \Lambda)$. Soit
\begin{equation} \label{equa:duas}
\duas: H^n\singc(-, \Lambda) \to H_n\singh(-, \Lambda)^\#
\end{equation}
le morphisme induit. Pour passer au cas $j > 0$ on utilisera les isomorphismes canoniques
\begin{align} \label{equa:trCH}
\trCH:   CH^{d + j}(X, n; \Lambda) &\stackrel{\sim}{\to} CH^{d + j} (\AA^j_X, n; \Lambda) \\
\trEt: \qquad H_c^{n}(\AA^j_X, \Lambda) &\stackrel{\sim}{\to} H_c^{n - 2j}(X, \Lambda(-j)) \label{equa:trEt}
\end{align}
(le choix de la notation $\trCH$ sera justifié plus tard) de \cite[Theorem 2.1]{Blo86} et \cite[XVIII.2.8.1]{SGA43} respectivement.

Choisissons une immersion ouverte $j: X \to \overline{X}$ de $X$ dans un schéma $\overline{X}$ projectif sur $k$. Soit $i: Y \to \overline{X}$ l'immersion fermée complémentaire. On choisit aussi une résolution injective $\Lambda \to \cI^\ast$ dans%
\footnote{La $h$-topologie (sur $\Sch / k$) est engendrée par la topology de Zariski et les morphismes propres surjectifs. On rencontrera aussi la $\tqfh$-topology (sur $\Sch / k$) qui est engendrée par la topologie de Zariski, et les morphismes finis surjectifs. On utilisera la même symbole pour un groupe abélien et son faisceau constant associé (le Zariski faisceau associé est déjà un $h$ faisceau).
} %
$\Shv_h(\Cor / k, \Lambda)$.%
\footnote{Le choix de $\ZZ / m$-coéfficients pour $\cI^\ast$ n'est pas nécessaire maintenant, mais ça rendra les choses plus facile plus tard cf. la Remarque~\ref{rema:sectionCanonique}.} %
Définissons le complexe%
\footnote{On utilisera souvent la même notation pour un objet, et le préfaisceau qu'il represente.}
\begin{equation} \label{equa:Xcdef}
[X^c] = \biggl ( [Y] \to [\overline{X}] \biggr )
\end{equation}
dans $\Comp(\PreShv(\Cor / k))$ concentré en degrés $-1$ et $0$.

\begin{theo}[Suslin-Voevodsky] \label{theo:premierChapeau}
Le complexe $\cI^\ast$ est acyclique en tant que complexe de faisceaux étales, et aussi en tant que complexe de $\tqfh$-faisceaux. En plus, les deux morphismes
\[ \underset{= Ext^n_\tqfh(\underline{C}_\ast z_jX, \Lambda)}{\hom \bigl (\underline{C}_\ast z_jX, \cI^\ast[n] \bigr )}
\stackrel{}{\to}
\underset{= H^n\singc(z_0X, \Lambda)}{\hom_{K(\Ab)}(C_\ast z_0X, \cI^\ast(k)[n]))} \]
\[ \underset{= Ext^n_\tqfh(\underline{C}_\ast z_0X, \Lambda)}{\hom \bigl (\underline{C}_\ast z_0X, \cI^\ast[n] \bigr )}
\to
\underset{= H_c^n(X, \Lambda)}{ \hom([X^c], \cI^\ast[n])} \]
sont des isomorphismes, où le premier est $\Gamma(k, -)$,  le deuxième est la composition avec le morphisme canonique $[X^c] = ([Y] {\to} [\overline{X}]) \to (z_0 Y {\to} z_0 \overline{X}) \to z_0(X) \to \underline{C}_\ast z_0X$. Les trois hom-groupes sont dans $K(\PreShv(\Cor / k))$.
\end{theo}

\begin{coro} \label{coro:lambdaCalcul}
Le morphisme canonique $\Lambda \to \cI^\ast(\Delta^\ast)$ est un quasi-isomorphisme.
\end{coro}

\begin{rema}
Les assertions d'acyclicité sont \cite[Corollary 10.7]{SV96} et \cite[Corollary 10.9]{SV96}. L'assertion que l'isomorphisme induit par $\Gamma(k, -)$ est un isomorphisme est \cite[Theorem 7.6]{SV96}, qui dit aussi que la composition avec $z_0(X) \to \underline{C}_\ast z_0X$ induit un isomorphisme. La composition avec $(z_0Y {\to} z_0\overline{X}) \to z_0X$ induit un isomorphisme puisque ce dernier est un quasi-isomorphisme après la $h$-faisceautisation \cite[Theorem 4.2.9, Theorem 4.3.1]{SV00}. La composition avec $[X^c] = ([Y] {\to} [\overline{X}]) \to (z_0Y {\to} z_0\overline{X})$ induit un isomorphisme parce que pour $W$ propre $[W] \to z_0W$ est un isomorphisme aprés la $h$-faisceautisation \cite[Proposition 4.2.14]{SV00} ($c\equi(W / S, 0) = z\equi(W / S, 0)$ quand $W \to S$ est propre).

Pour le corollaire, on utilise l'invariance par homotopie de la cohomologie étale.
\end{rema}

Donnons un nom à l'isomorphisme du Théorème~\ref{theo:premierChapeau}:
\begin{equation} \label{equa:isoCHsingcoh}
\isoCHsingcoh: H^n\singc(z_0X, \Lambda) \stackrel{\sim}{\leftarrow} \stackrel{\sim}{\to} H_c^n(X, \Lambda) 
\end{equation}

\begin{defi}[{}] \label{defi:susIso}
L'isomorphisme $\PhiS$ de Suslin est la composition
\[ \xymatrix@R=12pt{
H_c^{n - 2j}(X, \Lambda(-j))
\ar@{}[d]|{\viso} \ar@{}@<0.5ex>[d]^{\trEt} && \\
H_c^{n}(\AA^j_X, \Lambda) 
\ar@{}[r]|\cong \ar@{}@<0.5ex>[r]^{\isoCHsingcoh} & 
H^n\singc(z_0(\AA^j_X), \Lambda) 
\ar@{}[d]|{\viso} \ar@{}@<0.5ex>[d]^{\duas} && \\ &
H_n\singh(z_0(\AA^j_X), \Lambda)^\# 
\ar@{}[r]|\cong \ar@{}@<0.5ex>[r]^{\isoCHsing} & 
CH^{d + j}(\AA^j_X, n; \Lambda)^\#
\ar@{}[d]|{\viso} \ar@{}@<0.5ex>[d]^{\duas} \\ && 
CH^{d + j}(X, n; \Lambda)^\#.
} \]
où le morphisme $\trEt$ (resp. $\isoCHsingcoh$, $\duas$, $\isoCHsing$, $\trCH$) est de l'Equation~\ref{equa:trEt} (resp. \ref{equa:isoCHsingcoh}, \ref{equa:duas}, \ref{equa:CHsing}, \ref{equa:trCH}).
\end{defi}

%

\section{Le morphisme induit par la réalisation étale}

Dans cette section on discute le morphisme $\PhiR$.



Rappelons-nous qu'il y a une identification canonique de $H\singh_n(z_0X, \Lambda)$ avec le group $\hom(\ZZ [n], \underline{C}_\ast z_0 X \otimes \Lambda)$ de morphismes dans $K(\PreShv(\Cor / k))$. Par définition, $M(k)$ (resp. $M^c(X) \otimes \Lambda$) est l'image de $\ZZ$ (resp. $\underline{C}_\ast z_0 X \otimes \Lambda$) sous la projection canonique $K(\PreShv(\Cor / k)) \to DM\eff(k)$. Donc il y a un morphisme induit.
\begin{equation} \label{equa:isosingmot}
\isosingmot: H\singh_n(z_0X, \Lambda) \stackrel{\sim}{\to} \hom_{DM\eff(k)}(M(k), M^c(X)[n] \otimes \Lambda).
\end{equation}
C'est un isomorphisme après \cite[Theorem 8.1]{FV} (ou \cite[Theorem 5.4.19]{Kel12} pour la version sans supposer la résolution des singularités). Puis, il y a un isomorphisme canonique 
\begin{equation} \label{equa:mTr}
\Tr^{M}: M^c(\AA_X^j) \stackrel{\sim}{\to} M^c(X)(j)[2j]
\end{equation}
de \cite[Corollary 4.1.8]{Voev00} (où \cite[Theorem 5.5.9]{Kel12} pour la version sans supposer la résolution des singularités). Pour $j \geq 0$, l'isomorphisme $\isoCHmot$ de \cite[Proposition 4.2.9]{Voev00} est par définition la composition
\begin{equation} \label{equa:CHsingMot}
\begin{split}
\isoCHmot:  CH^{d + j}(X, n, \Lambda) &\underset{\trCH}{\stackrel{\sim}{\rightarrow}} 
CH^{d + j}(\AA^j_X, n, \Lambda) \\ &\underset{\isoCHsing}{\stackrel{\sim}{\leftarrow}}
H\singh_n(\AA^j_X, \Lambda) \\ & \underset{\isosingmot}{\stackrel{\sim}{\rightarrow}}
\hom_{DM\eff(k, \zpi)}(M(k), M^c(\AA^j_X)[n] \otimes \Lambda) \\ &\underset{\trM}{\stackrel{\sim}{\rightarrow}} 
\hom_{DM\eff(k, \zpi)}(M(k), M^c(X)(j)[2j + n] \otimes \Lambda).
\end{split}
\end{equation}
Si $j = 0$, le morphisme $\trCH$ et $\trM$ sont des égalités, d'où $\isoCHmot = \isosingmot \isoCHsing$.

On continue avec les $j: X \to \overline{X}$ et $i: Y \to \overline{X}$ choisis au-dessus. 
\begin{defi}
On rappelle qu'on avait défini dans (\ref{equa:Xcdef}) le complexe $[X^c]$ de $\Comp(\PreShv(\Sch / k))$. On définit aussi (les degrés sont indiqués en bas) 
\[ [X^c / m] = (\underset{-2}{[Y]} \stackrel{i - m \cdot id_Y}{\longrightarrow} \underset{-1}{[\overline{X}] \oplus [Y]} \stackrel{m \cdot id_{\overline{X}} + i}{\longrightarrow} \underset{0}{[\overline{X}]} ). \]
\end{defi}

\begin{rema} \label{rema:sectionCanonique}
Constatons que si $F^\ast$ est un complexe de préfaisceau de $\Lambda$-modules avec transferts, alors%
\footnote{Pour un complexe de préfaisceau $F^\ast$ et un complexe d'objets $W^\ast$, par $F^\ast(W^\ast)$ en veut dire $\Tot(F^q(W^p))$.}
 $F^\ast([X^c / m])$ devient égal à $F^\ast([X^c]) \oplus F^\ast([X^c])[1]$. D'où le morphisme canonique admet une section canonique
\begin{equation} \label{equa:section}
\sigma: F^\ast([X^c]) \to F^\ast([X^c / m])
\end{equation}
\end{rema}


\begin{theo} \label{theo:corMotIso}
Le morphisme évident
\[ \hom([\Delta^\ast][n], [X^c / m]) \stackrel{\sim}{\to} \hom([\Delta^\ast][n], z_0 X \otimes \Lambda[n] ) \]
est un isomorphisme.
\end{theo}

\begin{proof}
Vu l'isomorphisme (\ref{equa:isosingmot}) il suffit de montrer que $\hom([\Delta^\ast][n], [X^c / m]) \to \hom(\ZZ, M^c(X)\otimes \Lambda [n])$ est un isomorphisme. Pour $X$ propre, le morphisme $\hom([\Delta^\ast][n], [X]) \to \hom(M(k), M(X)[n])$ analogue est un isomorphisme par \cite[Theorem 8.1]{FV} (ou \cite[Theorem 5.4.19]{Kel12}). Ce cas implique le cas où $X$ n'est pas propre par le morphisme évident entre l'image du triangle distingué $[Y] \to [\overline{X}] \to [X^c] \to [Y][1]$ et le triangle distingué $M^c(Y) \to M^c(\overline{X}) \to M^c(X) \to M^c(Y)[1]$ (le dernier est distingué par \cite[Proposition 4.1.5]{Voev00} ou \cite[Proposition 5.5.5]{Kel12}). Puis le passage à $[X^c / m]$ et $M^c(X) \otimes \Lambda$ vient des triangles distingués évidents (dont le premier morphisme est multiplication par $m$).
\end{proof}

\begin{defi} \label{defi:phiIvo}
Définissons $\PhiR$ comme la composition
\begin{align*}
\hom_{D(k, \Lambda)}(\Lambda, R\pi_{X!} \Lambda_X(-j)[n-2j]) &\stackrel{\duaet}{\to} \hom_{D(k, \Lambda)}(R\pi_{X!} \Lambda_X(-j)[n-2j], \Lambda)^\# \\
&\stackrel{\RIvo}{\to} \hom_{DM\eff(k, \zpi)}(M(k), M^c(X)(j)[2j - n] \otimes \Lambda)^\# \\
&\stackrel{\isoCHmot}{\to} CH^{d + j}(X, n; \Lambda)^\#.
\end{align*}
où $\RIvo$ est le morphisme induit par $R_\Lambda$, les identifications canoniques (\ref{equa:realMotiCompSupp}) et (\ref{equa:realTwist}), et la section $\sigma: R\pi_{X!}\Lambda_X \to R_\Lambda(M^c(X) \otimes \Lambda)$ de l'Equation~\ref{equa:section}.
\end{defi}



\section{Demonstration de la compatibilité}

Dans cette section on montre les compatibilités affirmées dans l'introduction.

Constatons tout d'abord que le foncteur $\underline{C}_\ast: \PreShv \to \Comp(\PreShv)$ s'étend à $\Comp(\PreShv)$ en posant $(\underline{C}_\ast F^\ast)^n = \prod_{p + q = n} F^q(\Delta^{-p} \times -)$. Si l'on choisit correctement le signe des différentielles%
\footnote{Par exemple, le choix suivant marche. Pour $q - p = n$ on choisit la $(p,q)$-terme du différentiel $d^{n - 1}: (\underline{C}_\ast F^\ast)^{n - 1} \to (\underline{C}_\ast F^\ast)^n$ d'être  $(-1)^p d^{q - 1}_{\Delta^{-p}} + \sum_{j = 0}^{1 - p} (-1)^jF^q(\partial_j)$ où \mbox{$\partial_j: \Delta^{-p} \to \Delta^{1 - p}$} sont les cofaces de $\Delta^\ast$ et $d^{q - 1}_Y: F^{q - 1}(Y) \to F^q(Y)$ les différentiels de $F^\ast$ à l'objet $Y$. Du point de vu du bicomplèxe, nous avons inversé les différentiels verticaux sur les colonnes impaires.} il y a une inclusion canonique
\[ \epsilon: F^\ast \to \underline{C}_\ast F^\ast \]
du complexe $F^\ast$ dans le bicomplexe $\underline{C}_\ast F^\ast$ comme le colonne $p = 0$, qui est naturelle en $F^\ast$. La composition de $\underline{C}_\ast$ avec soi-même est donc munie de deux transformations naturelles canoniques
\begin{equation} \label{equa:twoAug}
\underline{C}_\ast \epsilon,\   \epsilon_{\underline{C}_\ast}: \underline{C}_\ast \rightrightarrows \underline{C}_\ast \underline{C}_\ast
\end{equation}

\begin{lemm} \label{lemm:eilenbergZilber}
Soit $F$ un préfaisceau avec transferts. Alors, il existe une homotopie entre les deux inclusions
$\underline{C}_\ast F \rightrightarrows \underline{C}_\ast \underline{C}_\ast F$.
\end{lemm}

\begin{rema}
Cette propriété semble légèrement plus forte que le fait que les deux inclusions sont des quasi-isomorphismes.
Le style de la démonstration est bien connu en topologie algébrique sous le nom ``Acyclic models'' \cite[Ch. VI, \S12]{Dol72}.
\end{rema}

\begin{proof}
C'est dû au fait que les $[\Delta^m]$ sont $\AA^1$-contractiles. Ajoutons quelques détails. Puisque $\underline{C}_\ast(-) = \uhom([\Delta^\ast], -)$ (hom interne dans $\Comp(\PreShv(\Cor / k)))$ il suffit de montrer que la différence des deux projections canonique\footnote{Par projections canoniques on veut dire les morphismes définis par les projections $\oplus_{p + q = n} [\Delta^p \times \Delta^q] \to [\Delta^n]$ vers les termes $p = 0$ ou $q = 0$.} $pr_1, pr_2: [\Delta^\ast \times \Delta^\ast] \rightrightarrows [\Delta^\ast]$ est homotope à zéro.%
\footnote{Le complexe $[\Delta^\ast]$ (resp. $[\Delta^\ast \times \Delta^\ast]$) est celui induit par Yoneda et le schéma simplicial $\Delta^\ast$ (resp. bisimplicial $\Delta^\ast \times \Delta^\ast$).} %
Puisque $pr_1$ et $pr_2$ sont égaux on degré zéro, on peut construire un homotopie par récurrence une fois qu'on sait que $H^n C_\ast\underline{C}_\ast [\Delta^n] = 0$ pour $n > 0$. Pour tout $F$, l'inclusion $C_\ast(\epsilon): C_\ast F \to C_\ast \underline{C}_\ast F$ est toujours un quasi-isomorphisme \cite[Corollary 7.5]{SV96}, et il nous suffit de voir que $H^n C_\ast [\Delta^n] = 0$. Cela est vrai pour n'importe quel schéma contractile\footnote{C'est-à-dire, pour les schèmas $\pi_Y : Y \to k$ tel qu'il existe un morphisme $\gamma: k \to Y$ et un morphisme $H: \AA^1_Y \to Y$ tel que $H i_0 = \id_Y$ et $H i_1 = \gamma \pi_Y$ où $i_0: Y \to \AA^1_Y$ est la section à zéro et $i_1: Y \to \AA^1_Y$ la section à 1.} pour les raisons classiques (\cite[Lemma 7.4]{SV96}), et c'est facile de voir que $\Delta^n$ est contractile.\footnote{Pour $H: \AA^1_{\Delta^n} \to \Delta^n$ on peut choisir, par exemple, le morphisme induit par $t_i \mapsto s t_i$ pour $i > 0$ et $t_0 \mapsto 1 - s + s t_0$ où les $t_i$ sont des coordonnées de $\Delta^n$ et $s$ une coordonnée de $\AA^1$.}%
%
\end{proof}

\begin{prop} \label{prop:pairing}
Les morphismes
\[ \begin{split}
\RIvo \circ \isosingmot: H_n\singh(X, \Lambda) &\to \hom(R\pi_{X!} \Lambda_X, \Lambda[n]) \textrm{ et } \\
\isoCHsingcoh: H^n\singc(X, \Lambda) &\cong \hom(\Lambda[n], R\pi_{X!} \Lambda_X)
\end{split} \]
sont compatibles avec les accouplements canoniques
\begin{align}
H_n\singh(X, \Lambda) &\times H^n\singc(X, \Lambda) \to \Lambda \textrm{ et } \tag{\ref{equa:singPair}} \\
\hom(\Lambda, R\pi_{X!} \Lambda_X) &\times \hom(R\pi_{X!} \Lambda_X, \Lambda) \to \Lambda. \tag{\ref{equa:etalPair}}
\end{align}
\end{prop}

\begin{rema}
On espère que le lecteur ne se perd pas dans les détails de la preuve. Tout simplement, il s'agit de représenter un élément de $H_n\singh(X, \Lambda)$ (resp. $H^n\singc(X, \Lambda)$) par un morphisme $\Delta^\ast[n] \stackrel{\alpha}{\to} \mathfrak{X}$ (resp. $\mathfrak{X} \stackrel{\phi}{\to} \cI^\ast[n]$) pour des interprétations de $\mathfrak{X}$ bien choisies dans une catégorie convenable, et constater que les deux accouplements ne sont que la composition. Les parties les plus difficiles sont de se frayer un chemin à travers tous les modèles de $\mathfrak{X}$, et d'avoir le Lemme~\ref{lemm:eilenbergZilber} sous la main.
\end{rema}

\begin{proof}
On utilise Théorème~\ref{theo:premierChapeau} et l'Equation~\ref{equa:isosingmot} pour représenter des éléments de $H_n\singh(X, \Lambda)$ et $H^n\singc(X, \Lambda)$ par des morphismes
\[ \ZZ[n] \stackrel{\alpha}{\longrightarrow} \underline{C}_\ast z_0X \otimes \Lambda \qquad \textrm{ et } \qquad \underline{C}_\ast z_0 X \stackrel{\phi}{\longrightarrow} \cI^\ast[n] \]
dans $K(\PreShv(\Cor / k))$. Comme on avait imposé à $\cI^\ast$ d'être $\Lambda$-linéaire, il y a une factorisation canonique $\underline{C}_\ast z_0 X \otimes \Lambda \stackrel{\psi}{\longrightarrow} \cI^\ast[n]$ et l'accouplement (\ref{equa:singPair}) n'est que l'image de la composition $\psi \alpha$ sous l'isomorphisme canonique $\Lambda \stackrel{\sim}{\to} \hom(\ZZ, \cI^\ast)$. Pour le comparaison on compose avec $\epsilon: \cI^\ast \to \underline{C}_\ast \cI^\ast$ et on définit
\[ (\alpha, \phi)\singc \stackrel{def}{=} \epsilon\, \psi\, \alpha \qquad \qquad \in \hom_{K(\PreShv(\Cor / k))}(\ZZ, \underline{C}_\ast \cI^\ast) \]

On construit maintenant le diagramme suivant.
\begin{equation} \label{equa:twoComp}
\xymatrix{
& [X^c] \ar[r] \ar[d] & z_0X \ar[r] \ar[d] & \underline{C}_\ast z_0 X \ar[r]^\phi \ar[d] & \cI^\ast \\
[\Delta^\ast] \ar[r]^-{\beta^\dagger} \ar@/_2em/[rr]_(0.2){\alpha^\dagger} & [X^c / m] \ar[r] \ar[urrr]|(0.32)\hole|(0.7)\hole_(0.85){\chi} & z_0 X \otimes \Lambda \ar[r] & \underline{C}_\ast z_0 X \otimes \Lambda \ar[ur]_{\psi} 
}
\end{equation}
L'image de $\phi$ dans $H^n_c$ est la composition avec le morphisme canonique $[X^c] \to \underline{C}_\ast z_0X$ (Théorème~\ref{theo:premierChapeau}). Pour l'accouplement on a besoin de sa factorisation canonique $\chi$ (Remarque~\ref{rema:sectionCanonique}). Sous l'adjonction $(-)^\dagger: \hom([\Delta^\ast], -) \cong \hom(\ZZ, \underline{C}_\ast(-))$ l'élément $\alpha$ correspond à un morphisme $[\Delta^\ast] \stackrel{\alpha^\dagger}{\to} z_0 X \otimes \Lambda$. Grace à l'isomorphisme du Théorème~\ref{theo:corMotIso} ce dernier se relève en un morphisme canonique $\beta^\dagger$, et via la section canonique $\cI^\ast([X^c]) \stackrel{\sigma}{\to} \cI^\ast([X^c / m])$ de l'Equation~\ref{equa:section} un élément $\cI^\ast(\beta^\dagger) \sigma \in \hom(\cI^\ast([X^c]), \cI^\ast([\Delta^\ast])) \cong \hom(R\pi_{X!} \Lambda_X, \Lambda)$.

Pour calculer l'accouplement (\ref{equa:etalPair}) de $\cI^\ast(\beta^\dagger) \sigma$ et $[X^c] \to \cI^\ast$, on compose $\beta^\dagger$ avec $\chi$ et on utilise l'isomorphisme canonique $\Lambda \stackrel{\sim}{\to} \hom([\Delta^\ast], \cI^\ast)$ (Corollaire~\ref{coro:lambdaCalcul}). L'adjoint de $\chi \beta^\dagger$ fournit la même élément de $\Lambda$, donc définissons
\[ (\alpha, \phi)\et \stackrel{def}{=} \underline{C}_\ast(\chi) \beta \qquad \qquad \in \hom_{K(\PreShv(\Cor / k))}(\ZZ, \underline{C}_\ast \cI^\ast). \]

%

L'égalité de $(\alpha, \phi)\singc$ et $(\alpha, \phi)\et$ découlera du diagramme suivant.
\begin{equation} 
\xymatrix{
\ZZ \ar[r]_-\eta \ar@/^1.5em/[rr]^(0.7){\beta} \ar@/^3em/[rrr]^{\alpha} & \underline{C}_\ast[\Delta^\ast] \ar[r]_{\underline{C}_\ast (\beta^\dagger)} & 
\underline{C}_\ast [X^c / m] \ar[r] \ar@/_0.8em/[drr]|(0.545)\hole_(0.3){\underline{C}_\ast \chi} & \underline{C}_\ast z_0 X \otimes \Lambda \ar[r]^-{\psi} \ar[d]^(0.45){\textrm{\danger}} & \cI^\ast \ar[d]^{\epsilon} \\ &&& 
\underline{C}_\ast \underline{C}_\ast z_0 X \otimes \Lambda \ar[r]_-{\underline{C}_\ast(\psi)} & 
\underline{C}_\ast\cI^\ast
}
\end{equation}
Rapellons-nous qu'il y a deux choix canoniques (Équation~\ref{equa:twoAug}) pour le morphisme marqué avec \danger. Pour que le carré commute (dans $\Comp(\PreShv(\Cor / k))$), il faut choisir $\epsilon_{\underline{C}_\ast}$. Par contre, si on veut que le chemin de bas soit égal à la flèche $\underline{C}_\ast \phi'$ (dans $\Comp(\PreShv(\Cor / k))$), il faut choisir $\underline{C}_\ast(\epsilon)$. Heureusement, ces deux choix sont égaux dans la catégorie homotopique (Lemme~\ref{lemm:eilenbergZilber}). D'où, dans $K(\PreShv(\Cor / k))$ nous avons $(\alpha, \phi)\singc = (\alpha, \phi)\et$.%
%
\end{proof}

\begin{prop} \label{prop:susEqualIvo}
Les morphismes $\PhiS$ et $\PhiR$ sont égaux. 
\end{prop}


\begin{proof}
C'est immédiat après la Proposition~\ref{prop:pairing} et le Lemme~\ref{lemm:realTrace} :

Dans le cas $j = 0$, par définition, $\PhiS$ est la composition du chemin supérieur dans le diagramme 
\[ \xymatrix{
\hom(\Lambda, R\pi_{X!} \Lambda_X[n]) \ar[d] & \ar@{-}[l]_-{\isoCHsingcoh}^\cong H^n\singc(z_0 X, \Lambda) \ar[d]^\viso \\
\hom(R\pi_{X!} \Lambda_X[n], \Lambda)^\# \ar[r]_-{\RIvo \circ \isosingmot} & H_n\singh(z_0X, \Lambda)^\# & \ar[l]^{\isoCHsing}_\cong CH^d(X, n; \Lambda)^\#
} \]
et $\PhiR$ est la composition du chemin inférieur dans ce diagramme. Le cas $j = 0$ n'est que la Proposition~\ref{prop:pairing} qui dit que le carré est commutatif.

Pour le cas $j > 0$, par définition nous avons ${\PhiS}_{, X} = \trCH {\PhiS}_{, \AA^j_X} \trEt$. Puisque on vient de voir que ${\PhiS}_{, \AA^j_X} = {\PhiR}_{, \AA^j_X}$ il suffit de montrer que ${\PhiR}_{, X} = \trCH {\PhiR}_{, \AA^j_X} \trEt$. Cela découle du diagramme suivant, dans lequel les colonnes sont les morphismes $\PhiR$, par définition.
\begin{equation*}
\xymatrix{
\hom(\Lambda, R\pi_{X!} \Lambda_X(-j)[n-2j]) \ar[d]_{\duaet}  %
& 
\ar[l]_-{\trEt} \hom(\Lambda, R\pi_{\AA^j_X!} \Lambda_X[n])  \ar[d]^{\duaet} %
\\
\hom(R\pi_{X!} \Lambda_X(-j)[n-2j], \Lambda)^\# \ar[d]_{\RIvo} \ar[r]^-{\trEt} & 
\hom(R\pi_{\AA^j_X!} \Lambda_{\AA^j_X}[n], \Lambda)^\#  \ar[d]^{\RIvo}  \\
\hom(\ZZ[n], M^c(X)(j)[2j] \otimes \Lambda)^\#  \ar[d]^{(\textrm{def.})}_\isoCHmot & \ar[l]_-{\trM} 
\hom(\ZZ[n], M^c(\AA^j_X) \otimes \Lambda)^\# \ar[d]^{\isoCHmot}  \\
CH^{d + j}(X, n; \Lambda)^\#  &
\ar[l]_{\trCH} CH^{d + j}(\AA^j_X, n; \Lambda)^\# 
}
\end{equation*}
La commutativité du carré en bas est tautologique (c'est la définition de $\isoCHmot$ donné dans l'Equation~\ref{equa:CHsingMot}). La commutativité du carré en haut n'est que l'associativité de la composition, et la commutativité du carré au centre est le sujet du Lemma~\ref{lemm:realTrace}.
%
%
%
\end{proof}

\begin{rema}
Malheureusement, on ne connait pas une référence pour la compatibilité de la réalisation étale avec les structures de dualité. Cette compatibilité est bien connue mais elle dépasse les objectifs de ce note.

Il découle facilement de cette compatibilité avec la dualité et la Proposition~\ref{prop:susEqualIvo} que quand $X$ est lisse l'isomorphisme de Suslin $\PhiS$ est le morphisme $\PhiBL$ de la conjecture de Beilinson-Lichtenbaum sous les identifications canoniques fournis par la dualité, i.e., le carré suivant est commutatif:
\[ \xymatrix@R=1em{
H_{c}^{n - 2j}(X, \Lambda(-j)) \ar@{}[r]|\cong^{\PhiS} \ar@{}[d]|\viso & 
CH^{j + d}(X, n; \ZZ / m)^\# \ar@{}[d]|\viso \\
H_{et}^{2(j + d) - n}(X, \Lambda(j + d))^\# \ar[r]_{\PhiBL} & 
H_{\mathcal{M}}^{2(j + d) - n}(X, \Lambda(j + d))^\# 
} \]
\end{rema}



\appendix

\section{Préfaisceaux de cycles relatifs} \label{sect:cyclesRelatifs}

Le préfaiseau $z_r(X)$ de \cite{Sus00} est, en fait, le préfaisceau $z\equi(X / k, r)$ de \cite{SV00}. Au lieu dele définir, ce n'est pas beaucoup plus difficil de se rappeler la théorie plus générale, ce qu'on fera maintenant. Remarquons que l'on donne une \emph{définition} et une grande partie des propriétés importantes des préfaisceaux de cycles relatifs et la catégorie des correspondances en moins de trois pages, même pour n'importe quels schémas noetherians et séparés. Il semble qu'il n'existe pas d'exposition aussi concise dans la littérature. D'habitude, la définition est l'aboutissement d'une quinzaine de pages de travail.

\subsection{Definition}
Dans \cite{SV00}, Suslin-Voevodsky définissent des préfaisceaux $z(X / S, r)$, $z\equi(X / S, r)$, $c(X / S, r)$, $c\equi(X / S, r)$ sur $\Sch / S$ pour $S$ un schéma noetherien, $X \in \Sch / S$ et $r \in \ZZ$. Soit $PreCycl(X / S, r)$ le groupe abélien libre engendré par les points $z$ de $X$ qui sont au-dessus d'un point générique de $S$, et de dimension $r$ dans leur fibre. Le problème, c'est que les groupes \mbox{$PreCycl(Y \times_S X / Y, r)$} pour $Y \in \Sch / S$ ne possèdent pas une fonctorialité en $Y$ raisonnable (par raisonnable on veut dire possédant les propriétés  \ref{rs:axiomRed} et \ref{rs:axiomComp} en-dessous). La solution est de travailler avec des sous-groupes des \mbox{$PreCycl(Y \times_S X / Y, r)$}. Pour un morphisme $f: Y \to X$ de $\Sch / S$ et $\sum n_i z_i$ une somme de points de $X$ on définit
\begin{equation} \label{equa:pullback}
 f^* \left (\sum n_i z_i \right ) = \sum n_i m_{ij} w_{ij} 
\end{equation}
où les $w_{ij}$ sont les points génériques de $Y \times_X z_i$ et $m_{ij} = \length \OO_{Y \times_X z_i, w_{ij}}$.

\begin{defi}
On peut définir le préfaisceau $z(X / S, r)$ comme le plus grand préfaisceau qui satisfait les conditions suivants : 
\begin{enumerate}
 \item \label{rs:axiom} Pour tout $Y \in \Sch / S$, le groupe $z(X / S, r)(Y)$ est un sous-groupe de $PreCycl(Y \times_S X, / Y, r)$. 
 \item \label{rs:axiomRed} Si $Y_{red} \to Y$ est l'inclusion canonique, $z(X / S, r)(Y) \to z(X / S, r)(Y_{red})$ est le morphisme évident.
 \item \label{rs:axiomComp} Soit $i: Y' \to Y$ un sous-schéma intègre de point générique $\iota: y \to Y$. On suppose que $\alpha = \sum n_i z_i \in z(X / S, r)(Y)$ soit tel que $y$ se trouve dans le lieu plat de $\amalg \overline{\{z_i\}} \to Y$. Alors $z(X / S, r)(i)(\alpha) = (\iota \times_S X)^*\alpha$ (vu comme cycle sur $Y' \times_S X$ via l'immersion $y \times_S X \subset Y' \times_S X$).
\end{enumerate}

Le préfaisceau $z\equi(X / S, r)$, (resp. $c(X / S, r)$, $c\equi(X / S, r)$) est par définition, le sous-préfaisceau de $z(X / S, r)$ des cycles $\sum n_i z_i \in z(X / S, r)(Y)$ tel que chaque $\overline{\{z_i\}} \to Y$ est équidimensionnel (resp. proper, équidimensionnel et proper).
\end{defi}

Soit $f: Y' \to Y$ un morphisme de $\Sch / S$. Pour alléger la notation, suivons Ivorra en écrivant $f^\circledast = z(X / S, r)(f)$ (Suslin-Voevodsky écrivent $cycl(f)$). Soit $i: \eta \to Y'$ l'inclusion des points génériques de $Y'$ et $\alpha \in z(X / S, r)(Y)$. Les axiomes \ref{rs:axiomRed} et \ref{rs:axiomComp} impliquent que $i^\circledast$ est injective, donc il suffit de connaitre $(fi)^\circledast(\alpha)$ pour connaitre $f^\circledast(\alpha)$. En particulier, si $f: Y' \to Y$ induit un isomorphisme sur un ouvert dense de $Y_{red}$, alors $f^\circledast$ est le morphisme évident induit par l'identification des fibres génériques de $Y \times_S X \to Y$ avec les fibres génériques de $Y' \times_S X \to Y'$. Par conséquence, avec ces axiomes et le théorème de platification de Raynaud-Gruson \cite[Proposition 2.2.2]{SV00} on peut calculer $f^\circledast$ pour n'importe quel morphisme de $\Sch / S$.

Si $S$ est regulier et $\iota: s \to S$ l'inclusion d'un point on peut aussi calculer $z(X / S, r)(\iota)$ avec la formule des $Tor$ de Serre quand $S$ est régulier \cite[Lemma 3.5.9]{SV00}.

\subsection{Caractérisation}
Un théorème (\cite[Corollaries 3.4.5]{SV}) affirme que pour $Y$ régulier l'inclusion $z_{equi}(X / S, r)(Y) \subseteq PreCycl(Y \times_S X / S, r)$ est une égalité. En général, on peut caractériser les cycles de $z(X / S, r)$ comme suit. Soit $\alpha = \sum n_i z_i \in PreCycl(Y \times_S X / Y, r)$ et $p: Y' \to Y$ un éclatement à centre rare. Puisque les fibres génériques de $Y' \times_S X \to Y'$ et $Y \times_S X \to Y$ sont isomorphes, on peut considérer $\alpha$ comme un élément de $PreCycl(Y' \times_S X / Y', r)$. Disons que $p$ \emph{platifie} $\alpha$ si $Y'$ est réduit, le centre de $p$ est rare, et chaque $\overline{\{z_i\}} \to Y'$ est plat (adhèrence dans $Y' \times_S X$). Dans ce cas, $\alpha \in z_{equi}(X / S, r)(Y')$ \cite[Corollary 3.3.11]{SV00}. Alors, un cycle $\sum n_i z_i \in PreCycl(Y \times_S X / Y, r)$ est dans $z(X / S, r)(Y)$ si et seulement si pour chaque diagramme commutatif
\[ \xymatrix{
Spec(K) \ar[d]_j \ar@<0.5ex>[r]^i \ar@<-0.5ex>[r]_{i'}  & Y' \ar[d]^p \\
y \ar[r] & Y
} \]
avec $y$ un point de $Y$, $p$ un éclatement qui platifie $\alpha$, et $K$ un corps, nous avons
\begin{enumerate}
 \item  $i^*\alpha = i'^*\alpha$, et
 \item \label{enum:integral} ce cycle est dans l'image de $j^*: z(X / S, r)(y) \to z(X / S, r)(K)$.
\end{enumerate}
Les morphismes $i^*, i'^*, j^*$ sont ceux définis dans l'Equation~\ref{equa:pullback}. 

\subsection{Fonctorialité}

Soit $f: X' \to X$ un morphisme de schémas. Pour $\sum n_i x_i$ une somme de points de $X$, on définit
\begin{equation} \label{equa:pushForward}
f_* \left (\sum n_i x_i \right )  = \sum n_i d_i f(x_i)
\end{equation}
où $d_i = [k(x_i), k(f(x_i))]$ si cette quantité est finie et zéro sinon. Avec cette définition, les préfaisceaux $c(X / S, r)$ et $c\equi(X / S, r)$ sont covariants en $X$ pour tout $S$-morphisme, et les préfaisceaux $z(X / S, r)$ et $z\equi(X / S, r)$ sont covariants en $X$ pour tout $S$-morphisme propre \cite[Corollary 3.6.3]{SV00}.

D'ailleurs, les morphismes de l'Equation~\ref{equa:pullback} induit une structure de contravariance en $X$ sur $z(X / S, r)$, $z\equi(X / S, r)$, $c(X / S, r)$ et $c\equi(X / S, r)$ pour les morphismes $f: X' \to X$ plat équidimensionnel (qui augmente $r$ par la dimension relative de $f$) \cite[Lemma 3.6.4]{SV00}.

Enfin, pour $F(-, -) = z(-, -)$, $z\equi(-, -)$, $c(-, -)$ ou $c\equi(-, -)$ il y a des morphismes de \emph{groupes}
\[ Cor(-, -): F(Y / X, m) \times F(X / S, n) \to F(Y / S, m + n) \]
définis comme suit. Pour $\alpha = \sum n_i z_i \in F(X / S, n)$ soient $\iota_i: \overline{\{z_i\}} \to Y$ les inclusions canoniques. Alors, $Cor(-, \alpha) = (\iota_i \times_X Y)_*\circ F(Y / X)(\iota_i)$ \cite[Corollary 3.7.5]{Sus00}.

Les morphismes $f_*, f^*$ et $Cor$ satisfont toutes les compatibilités qu'on souhaiterait \cite[Sections 3.6 and 3.7]{SV00}.

\subsection{Transferts}

Par définition, la catégorie $\Cor / S$ a un objet $[X]$ pour chaque objet de $\Sch / S$ et les morphismes sont $\hom_{\Cor / S}([Y], [X]) = c\equi(X / S, 0)(Y)$. La composition de $\alpha \in \hom_{\Cor / S}([Y], [X])$ et $\beta \in \hom_{\Cor / S}([X], [W])$ est
\[ \beta \circ \alpha = p_* Cor( q^*\beta , \alpha) \]
où $q: X \times_S Y \to Y$ et $p: X \times_S Y \times_S W \to W$ sont les projections canonique. Il y a un foncteur canonique $[-]: \Sch / S \to \Cor / S$ qui envoie un morphisme $f: Y \to X$ vers la somme $\sum \eta_i \in PreCycl(Y \times_S X / Y, 0)$ des points génériques de $Y$ vus comme points du graphe $\Gamma_f \subset Y \times_S X$ de $f$.

Un \emph{préfaisceau avec transferts} est un préfaisceau additif sur $\Cor / S$. Pour $\tau$ une topologie sur $\Sch / S$, un \emph{$\tau$-faisceau avec transferts} est un préfaisceau avec transferts tel que sa restriction à $\Sch / S$ est un $\tau$-faisceau. Il est alors clair que tout les préfaisceaux $z(X / S, r)$, $z\equi(X / S, r)$, $c(X / S, r)$ et $c\equi(X / S, r)$ sont des préfaisceaux avec transferts.

%

\section{La realisation étale}

Dans \cite{Ivo05}, Ivorra définit un foncteur monoïdal symétrique (dont le cible est la catégorie dérivée des faisceaux de $\ZZ / m$-modules sur le petit site étale de $k$, avec $m$ premier à la caractéristique de $k$)
\[ R_\Lambda: DM\gm\eff(k)\op \to D(k, \ZZ / m). \]
On rappelle que ce foncteur admet la description suivante. On choisit un complexe acyclique $\ZZ / m \to J^\ast$ dans $\Shv_{et}(\Cor / k, \ZZ / m)$ dont $H^p_{et}(Y, J^n) = 0$
pour tout $Y \in \Sch / k$, $p > 0$, et $n \geq 0$. Ivorra choisit la résolution de Godement mais on peut utiliser aussi le $\cI^\ast$ choisi au-dessus (Théorème~\ref{theo:premierChapeau}). Ce complexe induit alors un foncteur
\[ \Comp(\Cor / k)\op \to D(k, \Lambda);\quad  (\dots \to Y_{-1} \to Y_0 \to Y_1 \to \dots ) \mapsto \Tot J^\ast(Y_\ast) \]
qui se factorise par $DM\eff(k)$. Le lecteur peut consulter \cite{KS14} pour plus de détails. En particulier, puisque $J^\ast(Y)$ calcule la cohomologie étale pour chaque $Y \in \Sch / k$, le complexe $J^\ast([X^c])$ calcule la cohomologie à support compact (cf. Equation~\ref{equa:Xcdef}).

Or, le morphisme canonique $[X^c] \to z_0X$ devient un isomorphisme dans $DM\eff(k)$ (\cite[Proposition 4.1.5]{Voev00} ou \cite[Proposition 5.5.5]{Kel12} pour la version sans supposer la résolution des singularities). Par conséquence, il y a des identifications canoniques et naturelles
\begin{equation} \label{equa:realMotiCompSupp}
R_\Lambda(M^c(X)) = R\pi_{X!}\Lambda_X, \qquad \qquad R_\Lambda(M^c(X) \otimes \Lambda) = R\pi_{X!}\Lambda_X \oplus R\pi_{X!}\Lambda_X[1]
\end{equation}
où on écrit $\pi_{S}: S \to k$ pour le morphisme structural de $S \in \Sch / k$ (cf. la Remarque~\ref{rema:sectionCanonique} pour la deuxième).

Par définition, $\ZZ(1)[2]$ dans $DM\eff(k)$ est l'image du complexe $([\PP^1] {\stackrel{\pi_{\PP^1}}{\to}} [k])$ concentré en degré $0$ et $1$. Puisque le groupe %
$H^n_{et}(\PP^1, \Lambda)$ est nul pour $n \neq 0, 2$, et $H^0\et(k, \Lambda) \to H^0\et(\PP^1, \Lambda)$ est un isomorphisme, l'image de $\ZZ(1)[2]$ sous $R_\Lambda$ est $H^2(\PP^1, \Lambda)$. Alors, l'isomorphisme $H^2(\PP^1, \Lambda) \cong \Lambda(-1)$ induit une identification canonique
\begin{equation} \label{equa:realTwist}
R_\Lambda(\ZZ(1)[2]) = \Lambda(-1)[-2].
\end{equation}

\begin{prop} \label{lemm:realTrace}
Avec les identifications canoniques (\ref{equa:realMotiCompSupp}) et (\ref{equa:realTwist}) l'image du morphisme
$ \trM: \ZZ(j)[2j] \to M^c(X) $
de \cite[Corollary 4.2.4]{Voev00} sous la réalisation étale est $(-)(-d)[-2d]$ appliqué au morphisme 
$ \trEt: R\pi_{X!}\Lambda_{X}(d)[2d] \to \Lambda $
de \cite[XVIII.2.8.1]{SGA43}. 
\end{prop}

\begin{proof}
Puisque la réalisation est un foncteur monoïdal, il suffit de traiter le cas $X = k$.

Pour le morphisme $\trM$, considérons les morphismes de $\Comp(\PreShv(\Cor / k))$ suivants.
\[ \biggl ( \underset{0}{[\PP^1]} {\to} \underset{1}{[k]} \biggr )\ \stackrel{a}{\to}\ [\PP^1]\ \stackrel{b}{\to}\ \biggl ( \underset{-1}{[k]} {\to} \underset{0}{[\PP^1]} \biggr ) \ \stackrel{c}{\to}\ z_0\AA^1. \]
Par définition, l'image du premier objet dans $DM\eff(k)$ est $\ZZ(1)$, l'image du dernier est $M^c(\AA^1)$, et l'image de la composition est $\Tr^M$. Le morphisme $c$ est le morphisme qui induit l'identification (\ref{equa:realMotiCompSupp}).

Le morphisme $\trEt(-1)[-2]$ est par définition la composition
\[ R\pi_{\AA^1!}\Lambda\ \stackrel{d}{\to}\ R\pi_{\PP^1!}\Lambda\ \stackrel{e}{\to}\ \underline{H}^2R\pi_{\PP^1!}\Lambda[-2] \stackrel{f}{\cong} \Lambda(-1)[-2]. \]
Le morphisme $f$ est $(-1)[-2]$ du morphisme qui induit l'identification (\ref{equa:realTwist}).

La réalisation de $a$ (resp. $b$) est $e$ (resp. $d$). D'où, sous les identifications canoniques, l'image de $\trM$ et $\trEt(-1)[-2]$.
\end{proof}

\bibliographystyle{alpha}
\bibliography{bib}

\end{document}